\documentclass[preprint,12pt,3p]{elsarticle}

\usepackage{amssymb}
\usepackage{amsmath}

 \usepackage{amsthm}

 \newtheorem{theorem}{Theorem}
\newtheorem{corollary}{Corollary}
\newtheorem{lemma}{Lemma}

\journal{--}

\begin{document}

\begin{frontmatter}

\title{Bounds on prime gaps as a consequence of the divergence of the series of reciprocal primes}

%\tnotetext[label0]{This is only an example}

\author[label1]{Douglas Azevedo\corref{cor1}}
\address[label1]{UTFPR-CP- Caixa Postal 238,  Av. Alberto Carazzai, 1640, 863000-000, Cornelio Procopio - Brazil}
%\address[label2]{Address Two\fnref{label4}}

\cortext[cor1]{I am corresponding author}
%\fntext[label3]{I also want to inform about\ldots}
%\fntext[label4]{Small city}

\ead{dgs.nvn@gmail.com}
%\ead[url]{author-one-homepage.com}

%\author[label1,label5]{Author Three}
%\ead{author.three@mail.com}

\begin{abstract}
In this paper, using  the well known fact that the series of reciprocals of primes diverges, we obtain a general inequality for  gaps of consecutive primes that holds for infinitely many primes. As it is shown the key ingredient for this direct approach is  a consequence of the the Kummer's characterization of summable sequences of positive terms. Some interesting consequences are then presented. In particular, we show how the twin-prime conjecture is related to our main result.  
\end{abstract}

\begin{keyword}
%example \sep \LaTeX \sep template
gaps, primes, Kummer's test, Firoozbakht's conjecture.
\MSC[2010] 11A41, 28A20 .
\end{keyword}

\end{frontmatter}

% \linenumbers

\section{Introduction}

Let $p_n$ denote the $n$th prime number. The theory related to such numbers includes some of the most interesting  open problems and conjectures available in mathematics.
Among these open problems,  the topic that deals with the gap between consecutive prime numbers, that is, the one that  investigates the behaviour of the sequence $g_n=p_{n+1}-p_n$, has attracted attention of the  most prominent mathematicians in the area. 

 In (\cite{gold}) Dan  Goldston, J\'anos Pintz, and Cem Yildirim (also indicated as GPY)  presented a solution for 
  a long-standing open problem. They proved that there are infinitely many primes for which the gap to the next
prime is as small as we want compared to the average gap between consecutive primes. They showed that 
\begin{equation}\label{GPY}
\liminf_{n\to\infty} \frac{p_{n+1}-p_{n}}{\ln(p_n)}=0.
\end{equation}
There, the approach adopted is usually referred   as the
level of distribution of primes in arithmetic progressions and, with an additional assumption on this level of distribution the showed that
\begin{equation}\label{GPYa}
\liminf_{n\to\infty} p_{n+1}-p_{n}\leq 16.
\end{equation}
Latter, in \cite{gold1} the same authors considerably improved \eqref{GPY}
proving that 
\begin{equation}\label{GPY1}
\liminf_{n\to\infty} \frac{p_{n+1}-p_{n}}{\sqrt{\ln(p_n)}\ln(\ln(p_n))^2}<\infty.
\end{equation}
This result shows that there exist pairs of primes nearly
within the square root of the average spacing.

In 2013, Yitang Zhang \cite{zhang} published his celebrated paper providing  the first proof of finite gaps between prime numbers. In his work it is shown that
\begin{equation}\label{zhang}
\liminf_{n\to\infty} p_{n+1}-p_{n}\leq 7.10^7.
\end{equation}
His approach was a refinement of the  work of Goldston, Pintz
and Yildirim on the small gaps between consecutive primes \cite{gold} and 
a major ingredient of the proof is a stronger version of the Bombieri-Vinogradov
theorem \cite{zhang}. A nice exposition of the Zhang's proof can be found in \cite{gran}. 

The improvement of the Zhang's numerical bound on the gaps  was obtained right after. For instance, the work of Polymath8 (\cite{polymath8}) and  Maynard (\cite{maynard}) 
presented  a reduction of Zhang's bound to $4680$ and $600$, respectively.
In particular, in Maynard's work the  proofs involved quite different methods 
to Zhang and brought
the upper bound down to 600 using the Bombieri- Vinogradov Theorem (not
Zhang’s stronger alternative) and an improvement on GPY results. We refer to \cite{musson} and references there in for more information about the developments 
of the investigation about gaps of  prime numbers.

 The main objective of this paper is to present a general inequality (Theorem \ref{mainUP}) 
 involving gaps of infinitely many consecutive prime numbers which is our main result.
 In general, our result states  the inequality
 $g_n<p_n u_n$,
 holds for infinitely many (unknown) values of $n$, in which $\{u_n\}$ is a  sequence of positive numbers, given by $$u_n=\frac{q_{n+1}-q_n+1}{q_n},\,\,n\geq 1,$$
 where $\{q_n\}$ is any given sequence of positive numbers. 
 In particular, this result may be seen as an improvement of the Bertrand-Chebyshev theorem which state that for every $\epsilon>0$ the inequality  $g_n<p_n \epsilon$, holds for infinitely many values of $n$.

The reader will notice that the approach adopted to achieve  our main result also deserve a highlight  since it deals with quite elementary methods and it is based on the Kummer's characterization of the convergence of series of positive terms (\cite{tong}). 

   From our main result some interesting consequences are presented, in particular, we prove that the Firoozbakht's conjecture (\cite{ferreira, kourbatov}), that is, 
$$p_{n+1}^{\frac{1}{n+1}}<p_{n}^{\frac{1}{n}},\,\,n \in\mathbb{N},$$
 holds for infinitely many values of $n$. This conjecture  has been verified for all primes below $4.10^8$.    Although we are only getting improvements of basic results about gaps of prime numbers from elementary methods, we also indicate how our results may be related to the remarkable estimates \eqref{GPY}, \eqref{GPY1} and \eqref{zhang} obtained for gaps of prime numbers . Moreover, we also indicate how our main results are related to the twin-prime conjecture.

\section{Background}

In this section we present some basic results and notation that will be needed in order to state and prove our main theorem.  

Let us start with the following classical result. 
\begin{lemma}\label{primes}
The series $\sum \frac{1}{p_n}$ diverges.
\end{lemma}

The next result is a consequence of the prime number theorem (see \cite[p. 80]{apostol}). 
We will make use of the asymptotic notation $u(x)\sim v(x)$, meaning that
$$\lim_{x\to\infty}\frac{u(x)}{v(x)}=1.$$

\begin{lemma}\label{PNT}
We have the following asymptotic behaviour 
$$p_n\sim n\ln(n).$$
Consequently
$$\ln(p_n)\sim \ln(n\ln(n))\sim\ln(n).$$

\end{lemma}

%See, for instance, \cite[p.16]{apostol}, for a proof.

\begin{lemma}\label{rosser}(\cite{rosser})
For all positive integer $n$ we have that
$n\ln(n)<p_n$. Consequently, $\ln(n)< \ln(n\ln(n))<\ln(p_n)$ holds for all $n\geq 3$.
\end{lemma}

As will be presented, the key ingredient for our main theorem is a consequence 
of the Kummer's test for convergence of series of positive terms (\cite{ tong}).
The statement of this result is presented bellow and the main feature of if is that it characterizes the sequences of positive terms that are summable in the sense that it provides necessary and sufficient conditions for a positive sequence to be summable. For the convenience of the reader we include the proof of this important result.

\begin{lemma}\label{kummer1}
A sequence $\{a_n\}$  of positive real numbers is summable if and only if there exists a sequence  of positive numbers $\{b_n\}$, a real number   $c>0$ and a positive integer $n_0$ such that, for all $n>n_0$,  the inequality
$$b_n \frac{a_{n}}{a_{n+1}}-b_{n+1}\geq  c$$
holds .
\end{lemma}
\begin{proof}
Suppose that 
there exists a sequence  of positive numbers $\{b_n\}$, a real number   $c>0$ 
and a positive integer $n_0$ such that, for all $n>n_0$,  the inequality
$$b_n \frac{a_{n}}{a_{n+1}}-b_{n+1}\geq  c,$$
holds. This implies that
$$b_n a_{n}-b_{n+1}a_{n+1}\geq  c a_{n+1},\,\,n\geq n_0.$$
For each $k\geq 0$, 
this last inequality implies that
$$\sum_{n=n_0}^{n_0 +k}b_n a_{n}-b_{n+1}a_{n+1}\geq  c \sum_{n=n_0}^{n_0 +k} a_{n+1},$$
that is,
$$b_{n_0} a_{n_0}-b_{n_{0}+k+1}a_{n_{0}+k+1}\geq  c \sum_{n=n_0}^{n_0 +k} a_{n+1}.$$
Hence, 
$$b_{n_0} a_{n_0}\geq  c \sum_{n=n_0}^{n_0 +k} a_{n+1},$$
for all $k\geq 0$. This implies the convergence of the  series $\sum a_{n}$.

 Conversely, if $\sum a_n = M$ 
define 
$$q_n=\frac{M-\sum_{j=1}^{n}a_j}{a_n},\,\, n\geq 1.$$ 
For this sequence we have that
$$q_{n}\frac{a_{n}}{a_{n+1}}-q_{n+1}=\frac{M-\sum_{j=1}^{n}a_j}{a_{n+1}}-\frac{M-\sum_{j=1}^{n+1}a_j}{a_{n+1}}=1,$$
for all $n\geq 1$. The proof is concluded.

\end{proof}

Equivalently, Lemma \ref{kummer} may be rewritten as follows.
\begin{lemma}\label{kummer}
A sequence $\{a_n\}$  of positive real numbers is summable if and only if there exists a sequence  of positive numbers $\{b_n\}$ and a positive integer $n_0$ such that, for all $n>n_0$,  the inequality
$$b_n \frac{a_{n}}{a_{n+1}}-b_{n+1}\geq  1$$
holds .
\end{lemma}

As a consequence of Lemma \ref{kummer} we have the following lemma.

\begin{lemma}\label{lem}

Suppose that $\{a_n\}$ is a sequence of positive real numbers such that $\sum \frac{1}{a_n}$ diverges. 
Then, for every sequence of positive numbers $\{b_n\}$ and every $n_0\in\mathbb{N}$, there exists $n'>n_0$ such that
$$b_{n'} a_{n'+1}-b_{n'+1} a_{n'}< a_{n'}.$$
\end{lemma}
\begin{proof}
Suppose that there exists a sequence $\{b_n\}$ of positive real numbers and a 
$n_0\in\mathbb{N}$ such that
$$b_n a_{n+1}-b_{n+1} a_{n}\geq  a_{n}$$
for all $n>n_0$. 
This implies that 
$$b_n \frac{1/a_{n}}{1/a_{n+1}}-b_{n+1} \geq  1$$
for all $n>n_0$, and therefore, by Lemma \ref{kummer}, 
$\sum \frac{1}{a_n}$ would be convergent, but  this is a contradiction.

\end{proof}

\section{Main result and consequences}

The main result of this paper is as follows. The reader will note that it is a direct consequence of Lemma \ref{lem}. 

\begin{theorem}\label{mainUP}
For every  sequence $\{q_n\}$ of positive real numbers 
% there exists a subsequence $\{q_{n'}\}$
%for which
the inequality
$$q_{n} p_{n+1}-q_{n+1} p_{n}<p_{n},$$
holds for infinitely many values of $n$, sufficiently large.

\end{theorem}
\begin{proof}
If we take $a_n=p_n$ in Lemma \ref{lem},  since $\sum  \frac{1}{p_n} $ diverges, the proof follows the same lines the proof of Lemma \ref{lem}. 
\end{proof}

In what follows  $g_n=p_{n+1}-p_{n}$ will denote the $n$th gap between 
the consecutive primes $p_{n+1}$ and $p_n$.

From the previous Theorem \ref{mainUP} it is clear that, 
\begin{corollary}\label{coro}
For every sequence
$\{q_n\}$ of positive numbers, the inequalities 
%$\{q_n'\}$ for which  
$$(i)\,\,\, g_{n}<p_{n}\left(\frac{q_{n+1}-q_{n}+1}{q_{n}}\right),$$
and
$$ (ii)\,\,\,\frac{p_{n+1}}{p_{n}}<\frac{q_{n+1}+1}{q_{n}},$$
holds for infinitely many values of $n$.
\end{corollary}

For suitable choices of $\{q_n\}$ in Theorem \ref{mainUP} some interesting consequences are obtained, as it is shown below. 

\begin{corollary}
For infinitely many values of $n$ the inequality 
$$\frac{ g_{n}}{p_{n}}<\frac{2}{n},$$
holds. In particular
$$\liminf_{n\to\infty}\frac{g_{n} }{p_{n}}=0$$
and also
$$\liminf_{n\to\infty}\frac{n}{2p_{n}}g_{n}\leq 1.$$
\end{corollary}
\begin{proof}
Define $q_n=n$ in Corollary \ref{coro}.
\end{proof}

If $\{q_n\}$ is a sequence of positive numbers, let us write 
\begin{equation}\label{Qn}
Q_n:=p_n\frac{ q_{n+1}-q_{n}+1}{q_n},\,\,n\geq 1.
\end{equation}
We will also write $\Xi(\{q_n\})$ to denote the subset of $\mathbb{N}$,  of indexes $n$  for which 
$g_n<Q_{n}$, as indicated in Corollary \ref{coro}-$(i)$.

The following  result is  related to the remarkable estimates \eqref{GPY} and \eqref{GPY1},  obtained in \cite{gold} and \cite{gold1}, respectively. These estimates  were obtained through technical and deep results from analytic number theory. As we indicate, for suitable choices of $\{q_n\}$ in Corollary \ref{coro} it may be possible to obtain such estimates 
however, the effort now is directionated to obtain the sequence $\{q_n\}$ and the set $\Xi(\{q_n\})$.

%%%%%%%%%%%%%%%%%%%%%%%%%%%%%%%%%%%%%%%%%%%%%%%%%%%%%%%%%%%%%%%%%%%%%%%%%%%%%%%%%%
\begin{corollary}
$(i)$ 
If there exists a sequence $\{q_n\}$ of positive numbers such that 
$g_n<Q_n$ for $n\in\Xi(\{q_n\})$ and 
$$\lim_{n\in \Xi(\{q_n\})} n \,\frac{ q_{n+1}-q_{n}+1}{q_n}=0$$
then
$$\liminf_{n\to\infty}\frac{g_{n} }{\ln(p_n)}=0.$$

$(ii)$ If there exists a sequence $\{q_n\}$ of positive numbers such that 
$g_n<Q_n$ for $n\in\Xi(\{q_n\})$ and 
$$\lim_{n\in \Xi(\{q_n\})} \frac{n\ln(n)^{1/2}}{\ln(\ln(n))^{2}} \,\frac{ q_{n+1}-q_{n}+1}{q_n}< \infty$$
 then 
$$\liminf_{n\to\infty}\frac{g_{n} }{\ln(p_{n})^{1/2}\ln(\ln(p_n))^2}<\infty .$$
\end{corollary}

%%%%%%%%%%%%%%%%%%%%%%%%%%%%%%%%%%%%%%%%%%%%%%%%%%%%%%%%%%%%%%%%%%%%%%%%%%%%%%%%%%

\begin{proof}
%For $(i)$,  
From Corollary \ref{coro}-$(i)$  it follows that 
for such  $\{q_n\}$ the inequality
$$\frac{g_{n}}{\ln(p_n)}<\frac{p_n}{n\ln(n)}\frac{\ln(n)}{\ln(p_n)}n\frac{q_{n+1}-q_n +1}{q_n}$$
holds for all the infinitely many values $n$ in $\Xi(\{q_n\})$. Thus, if 
$$\lim_{n\in\Xi(\{q_n\})} n \,\frac{ q_{n+1}-q_{n}+1}{q_n}=0$$
an application of Lemma \ref{PNT} in the last inequality implies that
$$\liminf_{n\to\infty} \frac{g_n}{\ln(p_n)}=0.$$

For $(ii)$, the same idea applies since, from Corollary \ref{coro}, for every $\{q_n\}$  
$$\frac{g_{n}}{\ln(p_n)^{1/2}\ln(\ln(p_n))^{2}}<\frac{n\ln(p_n)^{1/2}}{\ln(\ln(p_n))^2}\frac{q_{n+1}-q_n +1}{q_n},$$
holds for all infinitely many values $n\in\Xi(\{q_n\})$.

\end{proof}

\noindent{\it Remark.} It is important to note that for any sequence $\{q_n\}$ of positive numbers, we have that $n Q_n>1$ for infinitely many values of $n$. This may been seen as a consequence of Lemma \ref{lem} with $a_n=n$. That is, since $\sum \frac{1}{n}$ diverges, we conclude that, for every sequence $\{q_n\}$ of positive terms the inequality 
$$1<n\frac{q_{n+1}-q_n+1}{q_n},$$
holds for infinitely many values of  $n$. 

Note also  that from our results we are able to extract the following result related to \eqref{GPY1} via much more elementary arguments.

\begin{corollary}
For every $\epsilon>0$
$$\liminf_{n\to\infty}\frac{g_{n} }{\ln(p_n)^{1+\epsilon}}=0.$$
\end{corollary}
\begin{proof}
Let $\epsilon>0$. Again, from Corollary \ref{coro}-$(i)$  it follows that 
for such  $\{q_n\}$ the inequality
$$\frac{g_{n}}{\ln(p_n)^{1+\epsilon}}<\frac{p_n}{n\ln(n)}\frac{n\ln(n)}{\ln(p_n)^{1+\epsilon}}\frac{q_{n+1}-q_n +1}{q_n}$$
holds for all the infinitely many values $n$ in $\Xi(\{q_n\})$. If we take $q_n=n$, for all $n\geq 1$ 
then
$$\frac{g_{n}}{\ln(p_n)^{1+\epsilon}}<2\frac{p_n}{n\ln(n)}\frac{\ln(n)}{\ln(p_n)^{1+\epsilon}}$$
for infinitely many values of $n$, hence, from Lemma \ref{PNT}
$$\liminf_{n\to\infty}\frac{g_{n}}{\ln(p_n)^{1+\epsilon}}\leq 2\liminf_{n\to\infty}\frac{p_n}{n\ln(n)}\frac{\ln(n)}{\ln(p_n)^{1+\epsilon}}=0.$$

\end{proof}

 Another interesting consequence of Theorem \ref{mainUP} which is related to the celebrated result of Zhang (\cite{zhang}) is presented below.

\begin{corollary}\label{CoroQn}
 If there exists a sequence $\{q_n\}$ of positive numbers such that the sequence $\{Q_n\}$, as in \eqref{Qn}, is nonincreasing, then $\liminf_{n\to\infty }g_n<\infty$.
\end{corollary}
\begin{proof}
If there exists such $\{q_n\}$, then
$g_n<Q_n$, for infinitely many values of $n\in\Xi(\{q_n\})$. 
But, since $Q_n>0$ for all $n$, the monotone convergence theorem implies that $Q_n$ converges to a positive number $M\geq 2$,  since $g_n<Q_n$, for such indexes. Passing to the limit in this last inequality we conclude that $\liminf_{n\to\infty} g_n\leq M$.

\end{proof}

Note that a $\{q_n\}$ that satisfies Corollary \ref{CoroQn} may be obtained from the non-negative solution(s) of the second order non-linear recurrence relation 
$$ q_{n+2}\leq\left[\frac{p_n}{p_{n+1}}\left(\frac{q_{n+1}-q_{n}+1}{q_n}\right)+1\right]q_{n+1}-1.$$

The next corollary is related to the well known  Firoozbakht's conjecture (\cite{ferreira, kourbatov}). This conjecture asserts that the sequence $\{p_{n}^{1/n}\}$ is decreasing for all $n\geq 1$ and its been verified  for all primes below $4.10^{8}$. 

Let us first prove a technical lemma that will be needed.

\begin{lemma}\label{lemaLim}
$$\lim_{n\to\infty} p_{n}^{1/n}=1.$$
\end{lemma}
\begin{proof}
It follows from Lemma \ref{PNT} that
\begin{eqnarray*}
\lim_{n\to\infty}p_{n}^{1/n}=\lim_{n\to\infty}e^{\frac{\ln(p_{n})}{n}}=\lim_{n\to\infty}e^{\frac{\ln(n)}{n}}=1.
\end{eqnarray*}

\end{proof}

\begin{corollary}
For infinitely many values of $n$ the following inequality holds
$$\frac{p_{n+1}^{\frac{1}{n+1}}}{p_{n}^{\frac{1}{n}}}<\frac{\ln(n+1)+p_{n+1}^{-1+\frac{1}{n+1}}}{\ln(n)}.$$
That is
$$\liminf_{n\to\infty}\frac{p_{n+1}^{\frac{1}{n+1}}}{p_{n}^{\frac{1}{n}}}\leq 1.$$

\end{corollary}
\begin{proof}
In Corollary \ref{coro} if we take $q_n=p_{n}^{1-\frac{1}{n}}\ln(n)$ 
and apply it in $(ii)$ then we obtain the inequality
$$\frac{p_{n+1}^{\frac{1}{n+1}}}{p_{n}^{\frac{1}{n}}}<\frac{\ln(n+1)+p_{n+1}^{-1+\frac{1}{n+1}}}{\ln(n)},$$
for infinitely many values of $n$. Therefore, from Lemma \ref{lemaLim}
$$\liminf_{n\to\infty}\frac{p_{n+1}^{\frac{1}{n+1}}}{p_{n}^{\frac{1}{n}}}\leq 1.$$

\end{proof}

As a consequence of the previous corollary, following the  same ideas of the 
proof of \cite[Theorem 1]{kourbatov}, but for infinitely many indexes $n$,  we can conclude that:

\begin{corollary}\label{coroAux}
For infinitely many values of $n$ the following inequality holds
$$g_n<\ln(p_n)^{2}-\ln(p_n)-1,$$
with $n\geq 10$.
\end{corollary}

However, it follows from Corollary \ref{coro} 
that it possible to get a sharper bound than this one presented in Corollary \ref{coroAux}. 
For instance:

\begin{corollary}\label{sharpB}
For infinitely many values of $n$ the following inequality holds
$$g_n<(n+1)\ln(n+1)-n\ln(n)+1,$$
with $n$ large.
\end{corollary}
\begin{proof}
It is enough to take $q_n=n\ln(n)$ in Corollary \ref{coro} and apply Lemma \ref{PNT}. 
\end{proof}

Let $b\geq 1$. Now, in order to show that the bound presented in Corollary \ref{sharpB} is sharper that the one presented in Corollary \ref{coroAux}, that is,   
$$(n+1)\ln(n+1)-n\ln(n)+1<\ln(p_n)^{2}-\ln(p_n)-b,$$
for all $n$ sufficiently large, note that this last   
 is equivalent to the following inequality
$$\frac{(n+1)\ln(n+1)-n\ln(n)+1}{\ln(p_n)^{2}-\ln(p_n)-b}<1,$$
for $n$ sufficiently large and $\ln(p_n)^{2}-\ln(p_n)-b\neq 0$.

To see that this last inequality holds for all $n$ sufficiently large, we may apply
the limit to the left hand side of this inequality and obtain that
\begin{eqnarray*}
\lim_{n\to\infty} \frac{(n+1)\ln(n+1)-n\ln(n)+1}{\ln(p_n)^{2}-\ln(p_n)-b}&=&
\lim_{n\to\infty} \frac{\ln(n+1)}{\ln(p_n)}\frac{1+\frac{\ln\left[\left(\frac{n+1}{n}\right)^{n}\right]}{\ln(n+1)} +\frac{1}{\ln(n+1)}}{\ln(p_n)-1-\frac{b}{\ln(p_n)}}\\
&=&0,
\end{eqnarray*}
since from Lemma \ref{PNT} 
$$\lim_{n\to \infty}\frac{\ln(n+1)}{\ln(p_n)}=1 .$$ 

\section{Final remarks}

In this paper we presented an alternative and elementary method to deal with gaps of consecutive primes. The central idea is to use a consequence of the Kummer's test for convergence of series of positive terms and the divergence of the series of the reciprocals of the prime numbers. The main feature  is that the proposed method  demands  less technical efforts to obtain bounds for the gaps of consecutive primes. 
 However the task of finding  suitable auxiliary sequences $\{q_n\}$ may not be easy. This behaviour  is inherited from the Kummer's test, which is more theoretical-functional than practical.

Let us conclude the paper by relating our results to the twin-prime conjecture.
%, the celebrated result of Zhang \cite{zhang} and its several refinements, which tell us that 
%$$\liminf_{n\to \infty} \,g_n<\infty.$$

For instance, if we choose the sequence $\{q_n\}$ in Theorem \ref{mainUP}
as
$$q_n = \begin{cases} n\ln(n), & \mbox{if } n\mbox{ is even} \\ (n-1)\ln(n), & \mbox{if } n\mbox{ is odd} \end{cases}
$$
then, for infinitely many values of $n$ the inequality
\begin{equation}
\label{twin}
    g_{n}< p_n \frac{q_{n+1}-q_n+1}{q_n}
\end{equation}
holds. In particular, note that, for $n$ even 
$$\lim_{n\to\infty} p_n \frac{q_{n+1}-q_n+1}{q_n}=2.$$
That is, if  $\Xi(\{q_n\})$ has infinitely many even numbers then  $g_n=2$ for all $n\in \Xi(\{q_n\})$ .
Hence, if \eqref{twin} holds for infinitely many even values 
of $n$ then the twin-prime conjecture would be true.  However, for $n$ odd we have that
$$\lim_{n\to\infty} p_n \frac{q_{n+1}-q_n+1}{q_n}= \infty,$$
which gives no information about the twin-prime conjecture.


\begin{thebibliography}{0}

\bibitem{apostol} Apostol, T.M., Introduction to Analytic Number Theory, Undergraduate Texts in Mathematics, Springer-Verlag, 1976.


\bibitem{ferreira} Ferreira, L.A. Mariano, H.L., Some consequences of the  Firoozbakht's conjecture. arXiv, 2017.


\bibitem{gold} Goldston, D.A,   Pintz, J., Yildirim,  C. Y. Primes in Tuples I, arXiv, 2005.

\bibitem{gold1} Goldston, D.A,   Pintz, J., Yildirim,  C. Y., Primes in Tuples II, arXiv, 2007.

\bibitem{gran}   Granville, A.  Bounded gaps between primes.
\texttt{http://www.dms.umontreal.ca/~andrew/CurrentEventsArticle.pdf} 2014. Last
accessed \mbox{04/07/2017}.

%\bibitem{mont} Montgomery. H. L. and Vaughan, R.C., Multiplicative Number Theory: I. Classical Theory, Cambridge studies in advanced mathematics, 2006.

\bibitem{kourbatov} Kourbatov, A., Upper bounds for prime gaps related to Firoozbakht's conjecture,(article 15.11.2) Journal of Integer Sequences, 2015.

\bibitem{maynard} Maynard, J.  Small gaps between primes. Annals of Mathematics, 181 2015 (1): 383-413.


\bibitem{musson} Musson, J.  Bounded Gaps Between Consecutive Primes. Dissertation, Trinity College, 2015.

\bibitem{polymath8} Polymath, D. H. J. New equidistribution estimates of Zhang type.
Algebra and Number Theory, 8:2067- 2199, 2014.

\bibitem{rosser}, Rosser, J. B. The n-th Prime is Greater than $n \log(n)$. Proceedings of the London Mathematical Society 45, 21- 44, 1939.

\bibitem{tong} Tong, J.  Kummer’s Test Gives Characterizations for Convergence or Divergence of all Positive Series. The American Mathematical Monthly, Vol. 101, No. 5 (1994), 450-452.

\bibitem{zhang} Zhang, Y.,  Bounded gaps between primes, Annals of Mathematics 179(3), 2014, 1121-1174.

\end{thebibliography}
\end{document}